\documentclass[10pt]{amsart}
\usepackage{amsfonts}
\usepackage{amsmath}
\usepackage{amssymb}
\usepackage{graphicx}%
\usepackage{dcolumn}
\usepackage{bm}
\newcommand{\beq}{\begin{equation}}
\newcommand{\eeq}{\end{equation}}
\newcommand{\beqa}{\begin{eqnarray}}
\newcommand{\eeqa}{\end{eqnarray}}
\newcommand{\nn}{\nonumber}
\newcommand{\noi}{\noindent}
\newtheorem{theorem}{Theorem}

\newtheorem{definition}[theorem]{Definition}

\newtheorem{proposition}[theorem]{Proposition}

\begin{document}
\title[Integrable maps and Galois differential algebras]{Integrable maps from Galois differential algebras, Borel transforms and \\   number sequences}

\author{Piergiulio Tempesta}
\address{Departamento de F\'{\i}sica Te\'{o}rica II (M\'{e}todos Matem\'{a}ticos de la f\'isica), Facultad de F\'{\i}sicas, Universidad
Complutense de Madrid, 28040 -- Madrid, Spain}

\begin{abstract}
A new class of integrable maps, obtained as lattice versions of polynomial dynamical systems is introduced. These systems are obtained by means of a discretization procedure that preserves several analytic and algebraic properties of a given differential equation, in particular symmetries and integrability \cite{Tempesta2}. Our approach is based on the properties of a suitable Galois differential algebra, that we shall call a Rota algebra. A formulation of the procedure in terms of category theory is proposed. In order to render the lattice dynamics confined, a Borel regularization is also adopted. As a byproduct of the theory, a connection between number sequences and integrability is discussed.
\end{abstract}
\date{May 5, 2013}
\maketitle

\tableofcontents

\vspace{3mm}

\section{Introduction}

The analysis of evolution equations on a discrete background is a very active research area, due to their ubiquitousness both in mathematics and in
theoretical physics. The formulation of quantum physics on a lattice is motivated, for instance,
by the need for regularizing divergencies in field theory \cite{wilson}, \cite{Lee}, \cite{MM} and in several scenarios of
quantum gravity \cite{ash}, \cite{KS}, \cite{thiemann}, \cite{RS}, \cite{GP}, where the discreteness of space--time geometry is assumed.
Discrete versions of nonrelativistic quantum mechanics have also been proposed (see, for instance, \cite{FL}, \cite{Z} and references therein).

Integrable discrete systems have been widely investigated as well. For several respects, they seem to be more
fundamental objects than the continuous ones.  For this reason, many efforts have been devoted to the construction
of discrete systems possessing an algebro--geometric structure reminiscent of that of continuous models. In particular, a challenging issue is to discretize nonlinear ordinary and partial differential equations in such a way that symmetry and integrability properties be preserved
(see the recent monograph \cite{Suris} for the treatment of the Hamiltonian point of view). Among the classical examples of discrete integrable systems we could mention the Toda systems and the Ablowitz--Ladik hierarchies \cite{AL1}--\cite{AL2}.

A rich literature exists on discrete differential geometry and related algebraic aspects \cite{BS1}, \cite{BS2}, \cite{Nov1}, \cite{Nov2}, \cite{Nov3}, \cite{LNP}, \cite{kono1}, \cite{Kuper1}, \cite{LW}, \cite{MV}, \cite{QN}, \cite{Ragnisco}, \cite{santini2}. Frobenius manifolds are relevant in the discussion of generalized Toda systems \cite{Carlet}, \cite{CDZ}. The problem of physically consistent discretizations has been considered in the context of field theories and Hamiltonian gravity, for instance in \cite{GP}, \cite{thooft}, \cite{RS}.

In the paper \cite{Tempesta2}, an approach offering a possible solution to this problem has been proposed. It is based on a technique developed in \cite{Ward}. The approach preserves both Lie symmetries and Lax pairs associated to a given nonlinear PDE. Also, it allows to construct new hierarchies of integrable discrete equations, representing discrete versions of the Gelfand--Dikii hierarchies. In this context, the notion of integrability is equivalent to the requirement that a large class of exact solutions can be analytically constructed. In particular, analytic solutions of continuous systems are converted into exact solutions of the corresponding systems defined on the lattice.

In the present work, by extending the results of \cite{Tempesta2}, \cite{Ward} we focus on the discretization of  vector fields of the form:
\begin{equation}
\frac{d}{dt}z= a_{N}z^{N}+a_{N-1}z^{N-1}+\ldots+a_{1}z+a_{0}, \label{ncont}
\end{equation}
with $N\in\mathbb{N}$, $z=z(t) \in \mathbb{R}$, $a_{0},\ldots, a_{N}\in \mathbb{R}$.

To this aim, a suitable Galois differential algebra, that we shall call a Rota algebra, is introduced. It is a pair $\left(\mathcal{F_{\mathcal{L}}},\mathcal{Q}\right)$, where $(\mathcal{F_L}, +,\cdot, *)$ is an associative algebra of formal power series over a lattice $\mathcal{L}$, endowed with an associative and commutative product "*", and $\mathcal{Q}$ is a delta operator that acts as a \textit{derivation} with respect to this product. The product has been first considered in [43], [6] and applied to the study of integrable models on the lattice. An analogous product has also been proposed, in a different context, in the theory of linear operators acting on spaces of polynomials \cite{ismail}.

The theoretical scheme we propose has a natural formulation in the language of Category Theory. We introduce the category $\mathcal{R} (\mathcal{F},\mathcal{Q})$ of Rota differential algebras and that of polynomial dynamical systems $\mathcal{E}$. The correspondence among continuous and discrete equations is expressed in terms of morphisms of algebras. A functor $F:  \mathcal{R} \rightarrow \mathcal{E}$ is defined. Essentially, it enables to map differential equations into difference equations by preserving the underlying differential structure: the symmetry and integrability properties of a continuous model are naturally inherited by the discrete ones associated and \textit{solutions are mapped into solutions}. This construction is explicitly discussed in the case when the delta operator $\mathcal{Q}$ is a forward difference operator (Theorem \ref{main}). Consequently, the algebraic structure underlying the dynamical models involved becomes transparent.

Other approaches existing in the literature deal with the discretization of differential equations on suitable nonlinear lattices, adapted to the symmetries of the problem (see, for instance, the review \cite{LW}). Instead, in this work we introduce a new family of integrable maps, defined on a lattice of equally spaced points. This kind of lattice is usually more convenient for physical applications, and is especially suitable for the implementation of algorithms of numerical integration.

The theory, being defined on a suitable associative algebra $\mathcal{F}$, is intrinsically a \textit{nonlocal} one: the value of a dynamical observable on the lattice depends on different points of the lattice. This is actually a very common feature observed in the theory of discrete integrable models.
When the dependence is on an infinite number of points, the approach still holds, but in a formal sense.
A regularization procedure should be adopted, in order to give a physical content to this case. This resembles very much what happens in nonlocal theories when dealing with operator product expansions.

However, in several cases, the value of an observable will depend on a \textit{finite} number of lattice points, and the discretization will be said to be \textit{effective}. With this terminology, we mean that the solutions obtained possess the general form
\begin{equation}
z_n= \sum_{k=0}^{n}a_k f(k),
\end{equation}
where $a_k \in \mathbb{R}$ or $\mathbb{C}$ and $f(k)$ is a function of the point $k$ on the lattice.

An interesting byproduct of our approach is a connection between discrete \textit{integrable systems} and \textit{number sequences}. Suppose we have a recurrence relation involving several points on a lattice. This recurrence relation defines in an auxiliary space of variables a discrete dynamical system, that we assume to be integrable. Once we obtain a particular solution of the discrete dynamical system, usually defined in terms of a number sequence,  we are able to construct a solution of the original recurrence relation. A different connection between number theory and difference operators, related to the theory of formal groups, has been established in \cite{Tempesta1}, \cite{Tempesta3}.

The necessary requirement to render the discretization procedure effective is that the points of the lattice  $\mathcal{L}$ we work over should correspond to the set of zeroes of the basic polynomials associated with $\mathcal{Q}$. Therefore, for higher order delta operators, nonlinear lattices should be used. A procedure to compute the basic polynomials has been proposed in \cite{DM}, \cite{LTW1}. In full generality, the problem of the determination of their zeroes in a closed form is essentially open \cite{LT}.



The theory developed here can be easily adapted to the study of physically relevant models, especially nonlinear lattice field theories, as for instance the Liouville theory. It can also be used to construct in a novel way quantum mechanical models on the lattice. The problem of a symmetry preserving discretization of the Schr\"odinger equation has been considered, for instance, in \cite{LTW1}, \cite{LTW2}. Our approach based on the use of Galois differential algebras of Rota type would provide a novel manner to solve the problem, possibly more suitable for the applications, since the correspondence with the continuum case is direct.


Another interesting open problem is to extend the proposed technique to the discretization of isochronous dynamical systems, in the spirit of the general framework proposed in \cite{calogero}.

The structure of the paper is the following. In Section II, an introduction to the algebraic techniques relevant for our discretization procedure is proposed. In Section III, the notions of Rota algebra and the categories $\mathcal{R} (\mathcal{F},\mathcal{Q}) $  and $\mathcal{E}$ are introduced and the main Theorem \ref{main} proved. In Section IV, some explicit examples of dynamical systems obtained according to the previous theory are constructed and some number theoretical aspects related to our approach are presented.

\section{Integrability--preserving discretizations of differential equations: a general framework}

\subsection{Basic definitions}

In this section we review some basic notions concerning the algebraic theory of polynomial sequences and of finite difference operators, much in the spirit of the monographs \cite{Roman}--\cite{Rota}.

Let us denote by $\{p_{n}\left(  x\right)\}_{n\in\mathbb{N}}$, $n=0,1,2,\ldots,$ a sequence of polynomials of order $n$ in a variable $x\in \mathbb{K}$, where $\mathbb{K}$ is a field of characteristic zero..
Let $\mathcal{F}$ denote the algebra of formal power series in $x$ endowed with the operations of sum and multiplication of series, and product of a series by a scalar. An element of $\mathcal{F}$ is expressed by a formal power series of the form
\begin{equation}
f\left(  x\right)  =\sum_{k=0}^{\infty}b_{k}x^{k}\text{.}%
\label{2.2}%
\end{equation}

Let $T$ be the shift operator, whose
action is given by $Tf\left(x\right)=f\left(x+\sigma\right)$, where $\sigma >0$. The operator $T$ can also be represented in terms of a differential operator as $T = e^{\sigma D}$, where $D$ denotes the standard derivative.
\begin{definition}
A linear operator $S$  is said to be {\it shift--invariant} if
commutes with $T$. A shift--invariant operator $\mathcal{Q}$ is called a \textit{delta operator} if $\mathcal{Q}x=const\neq0$.
\end{definition}
\noi We shall denote by $\mathcal{D}$ the set of delta operators.
\begin{definition}\label{def2}
A polynomial sequence $p_{n}\left(x\right)  $ is called a sequence of \textit{basic polynomials} for a delta operator $\mathcal{Q}$
if it satisfies the following conditions:%
\begin{align}
&  1)\text{ \ }p_{0}\left(  x\right)  =1;\notag\\
&  2)\text{ \ }p_{n}\left(
0\right)  =0\ \text{for all }n>0;\text{ \ }\nonumber\\
&  3)\text{ \ }\mathcal{Q} p_{n}\left(  x\right)  =np_{n-1}\left(  x\right).\nonumber
\end{align}
\end{definition}

Every delta operator $\mathcal{Q}$ has a unique sequence of associated basic
polynomials \cite{Rota}. Besides the standard derivative operator $\partial_x$,  a general class of difference operators can be defined by \cite{LTW1}
\begin{equation}
\Delta_{p}=\frac{1}{\sigma}\sum_{k=l}^{m}\alpha_{k}T^{k}\text{,}\quad l\text{,
}m\in\mathbb{Z}\text{,}\mathbb{\quad}l<m\text{,\quad}m-l=p\text{,}\label{2.9}
\end{equation}
where $\sigma$ can be interpreted as a constant lattice spacing and $\alpha_{k}\,$ are constants  such that
\begin{equation}
\sum_{k=l}^{m}\alpha_{k}=0\text{,}\quad\sum_{k=l}^{m}k\alpha_{k}=c\text{.}\label{2.10}
\end{equation}
and $\alpha_m\neq0$, $\alpha_l\neq0$. We choose $c=1$, to reproduce the derivative $\partial_x$ in the continuum limit. A difference operator of the
form (\ref{2.9}), which satisfies the equations (\ref{2.10}), is said to be a delta operator of
order $p$, if it approximates the continuous derivative up to terms of order $\sigma^p$.

As eq. (\ref{2.9}) involves $m-l+1$ constants $\alpha_{k}$, subject to just the two conditions (\ref{2.10}), we can fix all constants
$\alpha_{k}\,\ $ by choosing $m-l-1\,$ further conditions.

\subsection{Associative algebras and nonlocal functional products}

Let $\mathcal Q$ be a delta operator, and $\{p_n(x)\}_{n\in\mathbb{N}}$ be its associated basic sequence. Let us denote by $\mathcal{P}$ the space of polynomials in one variable $x$.  Since the basic polynomials $p_n(x)$ for every $\mathcal Q$ provide a basis of $\mathcal{F}$, any $f\in\mathcal{F}$ can be expanded into a formal series of the form
\begin{equation}
f(x)=\sum_{n=0}^{\infty}a_n p_n(x) \label{exp}.
\end{equation}

Let $\mathcal{L}$ be a lattice of equally spaced points on the real line. We shall denote by $\mathcal{F_L}$ the vector space of the formal power series defined on $\mathcal{L}$. We can endow the space $\mathcal{F}$ (and consequently $\mathcal{F_L}$) with the structure of an algebra, by introducing a suitable product related with the delta operator $\mathcal{Q}$ involved in the discretization considered. Precisely, in \cite{Tempesta2}, the following result has been proved.

\begin{theorem}\label{th1}
For any delta operator $\mathcal Q$, whose associated sequence is $\{p_n(x)\}_{n\in\mathbb{N}}$, the product
$$
*: \mathcal{P} \times \mathcal{P} \longrightarrow \mathbb{R}
$$
defined as
\begin{equation*}
p_n(x)*p_m(x):=p_{n+m}(x)
\end{equation*}
and extended by linearity is associative, commutative and satisfies the Leibnitz rule:
\begin{equation}
\mathcal{Q}\left(f(x)*g(x)\right)=\mathcal{Q}(f(x))*g(x)+f(x)*\mathcal{Q} (g(x)), \label{LR}
\end{equation}
where $f(x), g(x)\in\mathcal{F}$.
\end{theorem}

In the following, with a slight abuse of notation we will use the symbol "$*_{\mathcal{Q}}$" whenever we wish to emphasize the dependence of the "*" product on the choice of $\mathcal{Q}$.

One of the main features of the $*$ product is the fact that it is \textit{nonlocal}. Indeed, the product $f*g$ on a lattice $\mathcal{L}$ does depend on the values of $f$ and $g$ at distinct points of the lattice, with the exception of the standard point--wise product of functions on $\mathbb{R}$, that corresponds to the choice $\mathcal{Q}=\partial_x$.
The simplest version of this product, i.e. that associated with the forward difference operator $\Delta$, has been proposed in \cite{Ward}.

In general, the spaces $\mathcal{F}$ are infinite--dimensional. Nevertheless, the expansion (\ref{exp}) can truncate in some cases, that are of special interest for the applications.

\section{Category theory, differential equations and their discretization}

The discretization approach we wish to propose can be formulated in a natural way in the context of Category Theory. To this aim, we need to propose several new definitions.

\subsection{Rota algebras}
We introduce first the notion of a Rota algebra, as the natural Galois differential algebra over which the discretization procedure we are interested in is carried out. To motivate this definition, crucial for the subsequent discussion, we shall recall first that of Rota correspondence, that has been discussed extensively in the literature \cite{Rota}, \cite{LTW1}--\cite{LTW2}, \cite{DM}, \cite{LNO}. We will formulate it in the following way.

\begin{definition} \label{Rota}
We shall call the correspondence expressed by the following diagram
\begin{center}
\begin{eqnarray}\label{Rotadiag}
\nn \hspace{-16mm}  &\partial&  \mathop{\kern0pt \longrightarrow }\limits_{}^{\psi}\hspace{4mm} \mathcal{Q} \\
\nn &i\downarrow&  \hspace{9mm} i \downarrow \\
\qquad  &\{x^{n}\}_{n}&\mathop{\kern0pt \longrightarrow }\limits_{}^{\phi} \hspace{2mm} \{p_{n}(x)\}_{n}
\end{eqnarray}
\end{center}
the Rota (or umbral) correspondence. Here $i$ is the isomorphism between a delta operator $\mathcal{Q}$ and its associated basic sequence $\{p_n(x)\}_{n\in\mathbb{N}}$,  $\psi:\mathcal{D}\rightarrow\mathcal{D}$ and $\phi:\mathcal{P}\rightarrow\mathcal{P}$ are applications which transform delta operators into delta operators and basic sequences into basic sequences respectively.
\end{definition}

By means of the Rota correspondence we can associate with a linear differential equation (defined in the usual algebra of $\mathcal{C^{\infty}}$ or real analytic functions, endowed with the point--wise multiplication of functions) a linear difference equation in such a way that several algebraic properties are preserved. The application of this correspondence to the Schr\"odinger equation, for instance, allows to discretize it in such a way that integrals of motion are conserved \cite{LTW2}.

Nevertheless, in order to treat the more general case of \textit{nonlinear equations}, one needs to define their discrete versions on Rota algebras. Consequently, we shall formulate the Rota correspondence in the framework of category theory.

As an immediate consequence of Theorem \ref{th1}, we have the following result.
\begin{proposition}
For any choice of  $\mathcal{Q}$, the space $(\mathcal{F}, +,\cdot, *_{\mathcal{Q}})$ endowed with the sum of series, the multiplication by a scalar and the $*$ product, defined as in Theorem \ref{th1}, is an associative algebra.
\end{proposition}

\noi Therefore we can propose the following definition.

\begin{definition}
A Rota algebra is a Galois differential algebra $(\mathcal{F}, \mathcal{Q})$, where $(\mathcal{F}, +, \cdot, *_{\mathcal{Q}})$ is an associative algebra of formal power series, $\mathcal{Q}$ is a delta operator and $*$ is the composition law, defined as in Theorem \ref{th1}, such that $\mathcal{Q}$ acts as a derivation on $\mathcal{F}$:
\begin{equation}
i) \quad \mathcal{Q}(a+b)=\mathcal{Q}(a)+\mathcal{Q}(b),
\end{equation}
\begin{equation}
ii) \quad \mathcal{Q}(a*b)= \mathcal{Q}(a)*b+a*\mathcal{Q}(b).
\end{equation}
\end{definition}
\begin{proposition}
Given a delta operator $\mathcal{Q}$, there exists  a unique Rota algebra $(\mathcal{F}, \mathcal{Q})$ associated with it.
\end{proposition}
\begin{proof}
It follows from the uniqueness both of the "$*_{\mathcal{Q}}$" product associated with $\mathcal{Q}$ and of the sequence of polynomials $\{p_n(x)\}_{n\in\mathbb{N}}$, that ensures that, once $\mathcal{Q}$ is assigned, the algebra $(\mathcal{F}, +, \cdot, *_{\mathcal{Q}})$ is uniquely determined.
\end{proof}

\subsection{The Rota category}
We introduce a subcategory of the well--known category of associative algebras \cite{Mac Lane}: the \textit{Rota category}.
\begin{definition}
The Rota category, denoted by $\mathcal{R}(\mathcal{F}, \mathcal{Q})$ is the collection of all Rota algebras $(\mathcal{F}, \mathcal{Q})$, with morphisms defined by
\begin{equation*}
\mu_{\mathcal{Q},\mathcal{Q}'}: \mathcal{R}(\mathcal{F}, \mathcal{Q}) \longrightarrow \mathcal{R}(\mathcal{F}, \mathcal{Q})
\end{equation*}
\begin{equation}
(\mathcal{F}, +, \cdot, *_{\mathcal Q})\longrightarrow (\mathcal{F'}, +, \cdot, *_{\mathcal{Q'}})
\end{equation}
which are closed under composition.
\end{definition}
The action of the morphism $\mu_{\mathcal{Q},\mathcal{Q}'}$ on formal power series is defined by
\begin{equation}
\sum_{n} a_{n} p_{n} (x) \longrightarrow \sum_m a_{m} q_{m}(x), \label{morphism}
\end{equation}
where $\{p_{n}(x)\}_{n\in\mathbb{N}}$ and $\{q_{m}(x)\}_{m\in\mathbb{N}}$ are the basic sequences associated with $\mathcal{Q}$ and $\mathcal{Q}'$ respectively. The property of closure under composition is trivial.

We also introduce the category of \textit{polynomial dynamical systems}.

\begin{definition}
For any given choice of the coefficients $a_{0},\ldots,a_{N}\in\mathbb{R}$, the category $\mathcal{E}=\mathcal{E}\left(\{a_{0},\ldots, a_{N}\} \right)$ of polynomial dynamical systems is the collection of all equations of the form
\beq
eq(\mathcal{Q},z,*):= \mathcal{Q} z- a_{N}z^{*N}-a_{N-1}z^{*N-1}-\ldots-a_{1}z-a_{0}=0, \label{eq}
\eeq
where $z^{*N}:=\underbrace{z*\ldots *z}_{N-times}$ and $N\in\mathbb{N}$. The set of correspondences
\beq
u_{\mathcal{Q},\mathcal{Q}'}: \mathcal{E} \longrightarrow \mathcal{E},
\eeq
\beq
eq(\mathcal{Q},z,*) \longrightarrow eq(\mathcal{Q}',z,*') \label{morfeq},
\eeq
defines the class of morphisms of the category.
\end{definition}
\noi The closure of the morphisms $u_{\mathcal{Q},\mathcal{Q}'}$ under composition is easily verified. As a consequence of the previous construction, we can define a functor relating the categories of Rota differential algebras and abstract dynamical systems.
\begin{theorem} \label{functor}
The application
\[
F: \mathcal{R}(\mathcal{F},\mathcal{Q}) \longrightarrow \mathcal{E},
\]
\beq
(\mathcal{F}+, \cdot, *_{\mathcal{Q}}) \longrightarrow eq(\mathcal{Q},z,*),
\eeq
\beq
\mu_{\mathcal{Q},\mathcal{Q}'} \longrightarrow u_{\mathcal{Q},\mathcal{Q}'},
\eeq
is a covariant functor.
\end{theorem}
\begin{proof}
A direct verification shows that $F$ preserves the composition of morphisms:
\[
F(\mu_{\mathcal{Q}'',\mathcal{Q}'}\circ \mu_{\mathcal{Q}',\mathcal{Q}})=F(\mu_{\mathcal{Q}'',\mathcal{Q}'})\circ F(\mu_{\mathcal{Q}',\mathcal{Q}}).
\]
If we denote by $id_{\mathcal{Q}}:=\mu_{\mathcal{Q},\mathcal{Q}}$ the identity morphism, we also have
\[
F(id_{\mathcal{Q}}(\mathcal{A}))=id_{\mathcal{Q}}(F(\mathcal{A})),
\]
where $\mathcal{A}\in \mathcal{R}(\mathcal{F}, \mathcal{Q})$.
\end{proof}
This functor encodes all the main features of the discretization procedure we propose. It generalizes considerably the Rota correspondence (\ref{Rotadiag}).

\subsection{Main theorem}
The following Theorem \ref{main} represents the main result of this section. As a particular instance of the action of the functor $F$, it guarantees the discretization of a differential equation into a difference equation belonging to the same category $\mathcal{E}$, with the property that their solutions are mapped into each other.

\begin{theorem} \label{main}
\noindent Given a dynamical system of the form
\begin{equation}
\frac{d}{dt}z= a_{N}z^{N}+a_{N-1}z^{N-1}+\ldots+a_{1}z+a_{0}, \label{ncont2}
\end{equation}
where $z:\mathbb{R}_{+} \cup \{0\}\rightarrow\mathbb{R}$ is a $\mathcal{C}^{\infty}$ function, $a_{0},\ldots, a_{N} \in \mathbb{R}$, let
\begin{equation}
\mathcal{Q} z = a_{N}z^{*N}+a_{N-1}z^{*(N-1)}+\ldots+a_{1}z+a_{0} \label{abst}
\end{equation}
be the abstract equation associated with it, defined in the Rota algebra $(\mathcal{F}, \mathcal{Q})$ where $z\in (\mathcal{F}, +, \cdot, *_{\mathcal{Q}})$.
Let
\beq
z=\sum_{k=0}^{\infty} b_k t^k
\eeq
be a real solution of \eqref{ncont2} in the ring of formal power series in $t \in \mathbb{R}_{+} \cup \{0\}$. Then the equation

\begin{eqnarray}
\nonumber &z_{n+1}-z_{n}&= a_{N} \sideset{}{'}\sum_{k_1,\ldots,k_{N}=0}^{n} \frac{(-1)^{k_1+\cdots + k_{N}+n}}{k_1!\ldots k_{N}!}z_{k_1}z_{k_2}\cdots z_{k_N}\frac{n! (N-1)^{n-k_{1}-k_{2}-\cdots-k_{N}}}{(n-k_1-k_2-\ldots-k_{N})!}  \\
\nonumber &+& a_{N-1}\sideset{}{'}\sum_{k_1,\ldots,k_{N-1}=0}^{n}\frac{(-1)^{k_1+\cdots + k_{N-1}+n}}{k_1!\ldots k_{N-1}!}z_{k_1}z_{k_2}\cdots z_{k_N-1}\frac{n! (N-2)^{n-k_{1}-k_{2}-\cdots-k_{N-1}}}{(n-k_1-k_2-\ldots-k_{N-1})!}\\ \nonumber &+& \cdots +  a_2\sideset{}{'}\sum_{k_1=0}^{n}\sideset{}{'}\sum_{k_2=0}^{n}\frac{(-1)^{k_1+k_2+n}}{k_1!k_2!}z_{k_1}z_{k_2}\frac{n!}{(n-k_1-k_2)!} +a_1 z_n+a_0, \\ \noindent \label{Nmap}
\end{eqnarray}
that represents eq. \eqref{abst} for $\mathcal{Q}=\Delta^{+}$ on a regular lattice of points, indexed by the discrete variable $n$, admits as a solution the series
\beq
z_{n}=\sum_{k=0}^{n} b_k \frac{n!}{(n-k)!}. \label{part}
\eeq
Here we denote by $\sideset{}{'}\sum$ a sum ranging over all values of the indices $k_{i}$ such that $\sum_{i} k_{i}\leq n$.
\end{theorem}

\begin{proof}
The proof entails several steps. First, we prove that the equation \eqref{abst} converts into the discrete map \eqref{Nmap} for $\mathcal{Q}=\Delta^{+}$. To this aim, we shall define the function $z$ on an equally spaced lattice of points $\mathcal{L}$, indexed by $n\in\mathbb{N}$.. Also, we introduce an auxiliary space of variables, by means of a finite transformation on the lattice.
\begin{definition}
Let $p_n(t)$ be a basic sequence for a given delta operator $\mathcal{Q}$. We call the transformation
\begin{equation}
z(t)= \sum_{k=0}^{\infty}\widehat{z}_{k} p_{k} (t) \label{interpol}
\end{equation}
a \textit{discrete interpolating transformation} with coefficients $\widehat{z}_{k} \in\mathbb{R}$.
\end{definition}
\noindent In the specific case when $p_{k}(t)$ are the lower factorial polynomials, the discrete transform (\ref{interpol}) is \textit{finite}. Indeed, we have
\begin{equation}
p_{k}(n)= \begin{cases} 0 \qquad\qquad\qquad if \quad n<k, \\
\frac{n!}{(n-k)!}\qquad\qquad if\quad n \geq k. \label{pol}
\end{cases}
\end{equation}
\noindent It is straightforward to prove that
\begin{equation}
z_n=  \sum_{l=0}^{n} \frac{n!}{(n-l)!}\widehat{z}_l
\end{equation}
and, for the \textit{inverse interpolating transform},
\begin{equation}
\widehat{z}_n= \sum_{l=0}^{n} (-1)^{n-l} \frac{1}{  l!(n-l)!}z_l. \label{invtr}
\end{equation}
\noindent To derive the r.h.s. of eq. \eqref{Nmap}, we start computing explicitly the product $z^{*2}=z*z$. We get
\[
(z*z)_n=\sum_{l_1,l_2=0}^{\infty} \widehat{z}_{l_1}\widehat{z}_{l_2}P_{l_1+l_2}(n)=
\]
\[
\sum_{l_1=0}^{n}\sum_{l_2=0}^{n} \sum_{k_1=0}^{l_1} \sum_{k_2=0}^{l_2} (-1)^{l_1-k_1+l_2-k_2}[ \frac{z_{k_1}z_{k_2}}{k_1!(l_1-k_1)!k_2!(l_2-k_2)!}
\]
\[
\cdot \frac{n!}{(n-l_1-l_2)!}]=\sum_{k_1=0}^{n}\sum_{k_2=0}^{n} \frac{(-1)^{k_1+k_2}}{k_1!k_2!}z_{k_1}z_{k_2} \mathrm{K}_{n,{k_1},{k_2}},
\]

\noindent where we have introduced the kernel
\beqa
\mathrm{K}_{n,{k_1},{k_2}}= \sideset{}{'}\sum_{l_1, l_2=0}^{n}(-1)^{l_1+l_2}\frac{1}{(l_1-k_1)!} \cdot \frac{1}{(l_2-k_2)!}\frac{n!}{(n-l_1-l_2)!}. \label{kernel}
\eeqa
This expression, after some algebraic manipulations, reduces to
\begin{equation}
\mathrm{K}_{n,{k_1},{k_2}}= \sum_{l=k_1+k_2}^{n}\frac{(-1)^{l}n!}{(n-l)!}\frac{2^{l-k_1-k_2}}{(l-k_1-k_2)!}.
\end{equation}
Putting $s=l-k_1-k_2$, and summing over $s$, we arrive at the final expression for the kernel (\ref{kernel}):
\begin{equation}
\mathrm{K}_{n,{k_1},{k_2}}= \frac{(-1)^{n}n!}{(n-k_1-k_2)!}, \qquad n>k_{1}+k_{2}.
\end{equation}
By generalizing the previous reasoning, we prove the first statement of Theorem \ref{main}, i.e. that formula \eqref{Nmap} is the discrete analog of eq. \eqref{ncont2} on the regular lattice $t=n$.

The second part of the theorem, concerning the existence of the solution \eqref{part} for eq. \eqref{Nmap}, can be proved as follows. By means of the previous categorical approach, we have shown that eq. \eqref{Nmap} is nothing but the image of eq. \eqref{ncont2} under the action of the morphism $u_{\partial,\Delta^{+}}$ defined in $\mathcal{E}$. The morphism $\mu_{\partial,\Delta^{+}}$ provides the correspondence
\begin{equation}
\sum_{k} b_k t^k \longrightarrow \sum_{k} b_k p_{k}(n).
\end{equation}

By means of the action of the functor $F$, any formal series solution $z$ of eq. \eqref{abst} will correspond to a solution $z(t)$ of eq. \eqref{ncont2} and to a solution $z_n$ of eq. \eqref{Nmap}. In addition, on the lattice $\mathcal{L}$, the sum $\sum_{k} a_k p_{k}(n)$ truncates and converts into the finite sum \eqref{part}.

\end{proof}

\noi \textbf{Comment 1}.
The scheme offered by Theorem \ref{functor} can be used to generalize Theorem \ref{main} to infinitely many other integrable maps, each of them defined in a specific Rota algebra $(\mathcal{F},\mathcal{Q})$.

\section{Applications: recurrences, number sequences and \\ Borel regularization}

In this section, a connection between \textit{integrability} of discrete systems  and \textit{number sequences} will be established. It relies on the fact that a recurrence relation, defined on a suitable space of variables, can be associated with each of the discrete dynamical systems previously introduced. The solutions of these dynamical models are strictly related to those of the recurrences, that are expressed in terms of specific number sequences. For completeness, in the following we will study in detail some particular cases of this construction.

\subsection{A new family of quadratic dynamical systems}

Let $z=z(t):\mathbb{R}_{+} \cup \{0\}\rightarrow\mathbb{R}$ a $\mathcal C^{\infty}$ function. Consider the dynamical system
\begin{equation}
\frac{d}{dt}z=z^{2}. \label{cont}
\end{equation}
By applying Theorem \ref{main}, we discretize this model in such a way that the solutions of its discrete versions conserve the same structure of the solutions of \eqref{cont}. Consequently, we associate with the system (\ref{cont}) an abstract operator equation
\begin{equation}
\mathcal{Q} z =  z^{*2}. \label{syst}
\end{equation}
Here $z\in (\mathcal{F},+, \cdot, *_{\mathcal{Q}})$. In order to get a discrete map, we represent the equation (\ref{syst}) on a specific function algebra. We will choose $\mathcal{Q}=\Delta$. As customarily, we denote by $z_n$ the value of the function $z$ at the point $n$. Eq. (\ref{syst}) becomes on the lattice
\begin{equation}
z_{n+1}-z_{n} =  z^{*2}_n=\sideset{}{'}\sum_{k_1, k_2=0}^{n}\frac{(-1)^{k_1+k_2+n}}{k_1!k_2!}z_{k_1}z_{k_2}\frac{n!}{(n-k_1-k_2)!}, \label{dyn}
\end{equation}
which is a particular instance of eq. \eqref{Nmap}.

This system possesses an interesting \textit{alter ego} in the auxiliary space $\mathcal{A}$ of the variables $\widehat{z}$, defined in terms of a recurrence relation. Indeed, observe that
\begin{equation}
\Delta z_n= \sum_{k=0}^{\infty} \widehat{z}_l l p_{l-1}(n)=\sum_{l=0}^{n}\frac{n!}{(n-l)!}(l+1)\widehat{z}_{l+1}
\end{equation}
and
\begin{equation}
z^{*2}_n=\sum_{l_1,l_2=0}^{\infty}\widehat{z}_{l_1}\widehat{z}_{l_2}P_{l_1+l_2}=\sum_{l=0}^{n}\frac{n!}{(n-l)!}
\sum_{l_1=0}^{l}\widehat{z}_{l_1}\widehat{z}_{l-l_1}.
\end{equation}
We conclude that the following recurrence holds

\begin{equation}
(l+1)\widehat{z}_{l+1}=\sum_{l'=0}^{l}\widehat{z}_{l'}\widehat{z}_{l-l'}.
\end{equation}
This recurrence can be solved by using the ansatz
\begin{equation}
\widehat{z}_l=z_0^{l+1}.
\end{equation}
We deduce the solution
\begin{equation}
z_n=\sum_{l=0}^{n}\frac{n!}{(n-l)!} z_{0}^{l+1}, \label{solution}
\end{equation}

\noindent which can be also expressed in the compact form as
\begin{equation}
z_n=z_0 \cdot {}_{2} F_{0}(1,n;z_0).
\end{equation}
Therefore, the discretization procedure we adopted is \textit{effective}.
Also, it preserves integrability. Indeed, if we write the general solution of (\ref{cont}) in the functional space $(\mathcal{F},+, \cdot, *_{\Delta})$, we get the corresponding solution of the discrete model (\ref{dyn}).

To prove this, observe that the general solution of eq. (\ref{cont}) is provided by
\begin{equation}
z(t)=\frac{z_0}{1-z_0 t}.\label{solcont}
\end{equation}
We get

\begin{equation}
z(t)=z_0+z_0^{2}t+z_0^{3}t^{2}+\ldots \label{expandsol}
\end{equation}
and by virtue of the correspondence $t^n \longrightarrow p_n(t)$ defined by the morphism $\mu_{\partial,\Delta^{+}}$, taking also into account (\ref{pol}), we obtain eq. (\ref{solution}).

\subsection{Borel regularization}

The solution (\ref{solution}) of the discrete dynamical system (\ref{dyn}) diverges as $n\rightarrow \infty$. Instead, the solution (\ref{solcont}) of the continuous model (\ref{cont}) has a stable point at $z(\infty)=0$. In order to construct a system possessing localized dynamics, we introduce a \textit{finite Borel--type regularization} procedure.

\begin{definition}
Given a series $S\in\mathcal{F}$, with $S=\sum_{k=0}^{\infty}b_{k}z^k$, we call finite Borel--type regularization of $S$ the following transformation of the sequence of partial sums $\{S_{n}\}_{n\in \mathbb{N}}$ of $S$, i.e.
\beq
\mathcal{B}(S_{n})= \sum_{k=0}^{n} b_{k} \frac{{z^k}}{n!}.
\eeq

\end{definition}
\noindent Let
\beq
w_{n}=\mathcal{B}(z_{n})=\sum_{l=0}^{n}\frac{1}{(n-l)!} z_{0}^{l+1}.
\eeq
Consequently, the Borel regularized dynamical system associated with (\ref{dyn}) is
\beq
(n+1) w_{n+1}-w_{n}=\sideset{}{'}\sum_{k_{1}, k_{2}=0}^{n} (-1)^{k_{1}+k_{2}+n}w_{k_{1}}w_{k_{2}}\frac{1}{(n-k_1-k_2)!}\label{rnew}.
\eeq
This new dynamical system is still in correspondence with the original model (\ref{cont}). Indeed,  the Borel transform of the series obtained by applying to (\ref{expandsol}) the correspondence $t^n \rightarrow p_n(t)$ is a solution of eq. \eqref{rnew}.

\noindent Let us consider the dynamical system defined by differential equation
\beq
z'(t)=a_2 z(t)^2+ a_1 z(t)+a_0, \qquad a_0,a_1,a_2\in\mathbb{R}. \label{2pol}
\eeq

\noindent By integrating it, we obtain
\begin{equation}
z(t)=\frac{-a_1+\sqrt{\Gamma}\tan{\left\{\frac{1}{2}\sqrt{\Gamma}(t+c_0)\right\}}}{2a_2},
\end{equation}
where $\Gamma=4a_2a_0-a_1^{2}$, and $c_0\in\mathbb{R}$ is fixed by the initial condition.
\noindent By following the procedure described above, we obtain the discrete dynamical system

\begin{equation}
z_{n+1}-z_{n}= a_2 \sideset{}{'}\sum_{k_1, k_{2}=0}^{n}\frac{(-1)^{k_1+k_2+n}}{k_1!k_2!}z_{k_1}z_{k_2}\frac{n!}{(n-k_1-k_2)!}+ a_1 z_n+a_0, \label{newquad}
\end{equation}
with solution
\beq
z_{n}=\sum_{k=0}^{n} \beta_{k}\frac{n!}{(n-k)!},
\eeq
where $\{\beta_{k}\}_{k\in\mathbb{N}}$ is a number sequence, whose generating function is
\beq
\beta_{k}=\left[\frac{1}{k!}\frac{d^k}{dx^k}\frac{\sqrt{\Gamma}\tan{\left\{\frac{1}{2}\sqrt{\Gamma}(x+c_0)\right\}-a_1}}{2a_2}\right]_{\mid_{x=0}}.
\eeq
Its regularized version of Borel--type defines the map
\beq
(n+1) w_{n+1}-w_{n}=a_2\sideset{}{'}\sum_{k_{1}, k_{2}=0}^{n} (-1)^{k_{1}+k_{2}+n}w_{k_{1}}w_{k_{2}}\frac{1}{(n-k_1-k_2)!}+ a_1 w_n+\frac{a_0}{n!}\label{rnew2},
\eeq
with a solution given by the finite Borel--type transform of $z_n$:
\beq
w_n=\sum_{k=0}^{n}  \frac{\beta_{k}}{(n-k)!}.
\eeq

\noi \textbf{Comment 2}.
Besides the Borel regularization adopted in this work, other regularization procedures are possible (for instance, the Mittag--Leffler one). The other procedures  a priori could lead to discrete dynamical systems different than \eqref{rnew} or \eqref{rnew2}.

\subsection{A more general family of dynamical systems and associated recurrences}

\noindent Let us consider the dynamical system
\begin{equation}
\frac{d}{dt}z=a_{N} z^{N}, \qquad N\in\mathbb{N} \qquad a_{N}\in \mathbb{R}. \label{gendds}
\end{equation}
where $z:\mathbb{R}\cup\{0\}\rightarrow\mathbb{R}$ is a $\mathcal{C}^{\infty}$ function. The third system of the hierarchy is the dynamical system
\begin{equation}
\Delta z = a_3 z^{*3}. \label{3syst}
\end{equation}
Its realization in the space $(\mathcal{F},+, \cdot, *_{\Delta^{+}})$ can be obtained in a similar way. Therefore, eq. (\ref{3syst}) becomes on the lattice
\beq
z_{n+1}-z_{n}= a_3 \sideset{}{'}\sum_{k_1,k_2,k_3=0}^{n}\frac{(-1)^{k_1+k_2+k_3+n}}{k_1!k_2!k_3!}z_{k_1}z_{k_2}z_{k_3}\frac{2^{n-k_{1}-k_{2}-k_{3}}n!}{(n-k_1-k_2-k_3)!}. \label{3new}
\eeq
This system, as before, can be written in terms of the transformed variables (\ref{invtr}). It has the form of a recurrence relation
\begin{equation}
(l+1)\widehat{z}_{l+1}=a_3\sum_{\overset{l_1,l_2=0}{l_1+l_2\leq l}}\widehat{z}_{l_1}\widehat{z}_{l_2}\widehat{z}_{l-l_1-l_2}.
\end{equation}
A class of real solutions of eq. \eqref{gendds}, for $N=3$ is provided by
\begin{equation}
z(t)=\frac{1}{\sqrt{2}\sqrt{-a_3t+c_0}},
\end{equation}
where $a_{3} <0$ and $c_0\in\mathbb{R}_{+}$ is an arbitrary constant fixed by the initial condition. Since our discretization preserves integrability, we get the following series solution of the dynamical system \eqref{3new}:
\begin{equation}
z_n=\sum_{k=0}^{n}\frac{\gamma_k}{\sqrt{2}}\frac{n!}{(n-k)!}\frac{ a_3^{k}}{c_0^{(2k+1)/2}}, \qquad n=0,1,2,\ldots
\end{equation}
where the first terms of the sequence $\{\gamma_k\}_{k\in\mathbb{N}}$ are
\begin{equation}
\gamma_0=1,\gamma_1=\frac{1}{2},\gamma_2=\frac{3}{8},\gamma_3=\frac{5}{16}, \nn
\end{equation}
\beq
\gamma_4=\frac{35}{128}, \quad \gamma_5=\frac{63}{256},\quad etc.
\eeq
The generating function of this sequence is given by
\begin{equation}
\gamma_k= \left[\frac{d^k}{dx^k}\frac{1}{\sqrt{1-x}}\right]\mid_{x=0}. \label{seq}
\end{equation}

The procedure can be easily generalized to each value of $N\in \mathbb{N}$.

The Borel regularization approach, described above in the simpler case of quadratic dynamical systems, can be extended in a completely analogous way to the case \eqref{3new} or to the maps arising from more general polynomial vector fields. The number sequence $\{b_k\}_{k\in\mathbb{N}}$ appearing in the explicit solution of the maps and of their Borel--transformed analogs are preserved.

Consider now a general recurrence relation of the form

\beq
(l+1)\widehat{z}_{l+1}=a_{N} \sum_{\overset{l_1,l_2,\cdots,l_N=0}{l_1+\cdots+l_N\leq l}}\widehat{z}_{l_1}\widehat{z}_{l_2} \cdots \widehat{z}_{l_{N}}\widehat{z}_{l-l_1-l_2-\cdots l-N}. \label{recur}
\eeq
In order to find an exact solution of it, for every $N$, we interpret the recurrence (\ref{recur}) as the difference equation defining an abstract dynamical system, in the auxiliary space $\mathcal{A}$.
Then, a particular solution of the recurrence can be obtained by associating with it the discrete map

\beqa
\nn z_{n+1}-z_{n}= a_{N} \sideset{}{'} \sum_{k_1,\ldots,k_N=0}^{n}\frac{(-1)^{k_1+\cdots+k_N+n}}{k_1!\cdots!k_N!}z_{k_1}\cdots \\  \cdot z_{k_N}\frac{n! (N-1)^{n-k_{1}-k_{2}-\cdots-k_{N}}}{(n-k_1-k_2-\cdots k_N)!}, \label{discretemap}
\eeqa
obtained by means of the inverse transform (\ref{invtr}). Then, we come back to the continuous system (\ref{gendds}), and construct a general solution of it.

Once we expand this solution and discretize it on the lattice $\mathcal{L}$ by following the procedure described above, we deduce a particular solution of the map \eqref{discretemap} (which depends on the choice of the initial condition of (\ref{gendds})) in terms of a number sequence; from (\ref{invtr}) we construct a solution of the recurrence (\ref{recur}).

\section*{Acknowledgment}

I would like to thank heartily F. Calogero for his warm hospitality, very helpful advice,  discussions, encouragement and a careful reading of the manuscript.

I am grateful to  R. A. Leo and G. \'Alvarez Galindo for stimulating discussions and a careful reading of the manuscript. Useful discussions with B. Dubrovin, L. Martinez Alonso and M. A. Rodr\'iguez are also gratefully acknowledged.

I thank for kind hospitality the Dipartimento di Fisica, Universit\`a di Roma "La Sapienza", where part of this work has been carried out.

This research has been supported by the grant FIS2011--22566, Ministerio de Ciencia e Innovaci\'{o}n, Spain.

\end{document}